\documentclass[12pt]{amsart}

\usepackage[T1]{fontenc}
\usepackage{amssymb}
\usepackage{amsmath}
\usepackage{mathtools}
\usepackage{txfonts}
\usepackage{enumitem}
\usepackage[hidelinks]{hyperref}

\makeatletter
\@namedef{subjclassname@2020}{%
  \textup{2020} Mathematics Subject Classification}
\makeatother

\theoremstyle{plain}

\newtheorem{theorem}{Theorem}[section]
\newtheorem{proposition}[theorem]{Proposition}

\newtheorem{lemma}[theorem]{Lemma}

\theoremstyle{definition}

\newtheorem{remark}[theorem]{Remark}

\newcommand{\C}{\mathbb C}
\newcommand{\F}{\mathbb F}
\newcommand{\Z}{\mathbb Z}
\newcommand{\E}{\mathbb E}
\newcommand{\Id}{\mathbf 1}
\newcommand{\U}{\mathrm U}
\newcommand{\CC}{\mathcal C}
\newcommand{\PG}{\mathsf P}
\newcommand{\ket}[1]{\lvert #1\rangle}
\newcommand{\ip}[2]{\left\langle #1,#2\right\rangle}
\newcommand{\norm}[1]{\left\lVert #1\right\rVert}
\newcommand{\tr}{\operatorname{tr}}

\begin{document}

\baselineskip=17pt

%\title[Quantum uniformity measures as pullbacks]{Quantum uniformity measures are pullbacks of matrix-valued uniformity functionals}

\title[Weyl-orbit embedding and Clifford hierarchy]{The Weyl-orbit embedding and a representation-theoretic characterization of the Clifford hierarchy}

\author[A. Jamneshan]{Asgar Jamneshan}
\address{Faculty of Mathematics, TU Dresden, 01062 Dresden, Germany}
\email{asgar.jamneshan@tu-dresden.de}

\author[B. Ünal]{Berke Ünal}
\address{Department of Mathematics \\ Koç University \\ Istanbul, Turkey}
\email{bunal21@ku.edu.tr}

\date{\today}

\begin{abstract}
Bu, Gu, and Jaffe introduced the quantum uniformity measures to analytically characterize the Clifford hierarchy. 

We prove that the quantum uniformity measures are the pullbacks of the matrix-valued uniformity functionals of Gowers and Hatami under the map \(U\mapsto(a\mapsto W_aUW_a^*)\), which we call the \emph{Weyl-orbit embedding}. We study the structure of the Weyl-orbit embedding, and among other results, establish a canonical decomposition of the ambient matrix-valued uniformity space into uniformity preserving twisted copies of the quantum uniformity space. 

The Weyl-orbit embedding identifies the Clifford hierarchy with Leibman polynomial maps. This identification yields a representation theoretic characterization of the Clifford hierarchy in terms of unitary representations of the universal polynomial groups introduced by the first author and Thom.
\end{abstract}



\maketitle

The Clifford hierarchy was introduced by Gottesman and Chuang \cite{GC99} and plays an important role in fault-tolerant quantum computation. It is the nested sequence $$\CC^{(1)}\subseteq \CC^{(2)}\subseteq\ldots$$ defined recursively by taking $\CC^{(1)}$ to be the Pauli group and, for $k\geq2$, $$\CC^{(k)}=\{U:U\CC^{(1)}U^*\subseteq \CC^{(k-1)}\}.$$ Thus $\CC^{(2)}$ is the Clifford group, while already in the qubit case the higher levels need not be groups, starting at the third level \cite{anderson}. The lack of an algebraic structure is a serious obstacle in e.g.~determining tolerant testers for the higher levels of the Clifford hierarchy \cite{BBCH26}. Despite extensive efforts, a general algebraic classification of the higher levels of the Clifford hierarchy remains elusive~\cite{deSilvaLautsch:2026}.

Recently, Bu, Gu, and Jaffe \cite{BGJ25} introduced the quantum uniformity measures $\|\cdot\|_{Q^k}$. They proved that, for $k\geq1$, a unitary $U$ belongs to $\CC^{(k)}$ if and only if $\|U\|_{Q^{k+1}}=1$, with membership in the first level understood up to scalar phase. We recall the relevant definitions in Section~\ref{sec:prelim}.

Fix a prime local dimension $p$, let $n\geq 1$ be an integer and $N=p^n$, set $A\coloneqq \F_p^n\times \F_p^n$, and let $(W_a)_{a\in A}$ denote the Weyl operators. The Weyl-orbit embedding associates to a linear operator $T\in \mathrm{M}_N(\C)$ a function $F_T\colon A\to \mathrm{M}_N(\C)$ defined by $$F_T(a)\coloneqq W_a T W_a^*.$$ We show that the Weyl-orbit embedding preserves uniformity, that is, for any $k\geq 1$, $\|T\|_{Q^k}=\|F_T\|_{U^k(A)}$, where $\|\cdot\|_{U^k(A)}$ are the uniformity functionals on matrix-valued functions on groups introduced by Gowers and Hatami \cite{GH17}. We prove these results in Section \ref{sec:main}.

The norm properties of these measures/functionals exhibit an interesting distinction. In a recent work, the second author \cite{Unal26} proved that $\|\cdot\|_{Q^4}$ is a norm (the case $k=2,3$ was established earlier in \cite{BGJ25})), however the triangle inequality fails for $\|\cdot\|_{Q^k}$ for all $k\geq 5$ in general. He also showed that the matrix-valued uniformity functionals of Gowers and Hatami fail the triangle inequality for every $k\geq3$ in general. Thus, at orders $k=3,4$, the Weyl-orbit embedding identifies subspaces on which the matrix-valued uniformity functionals define genuine norms, even though they need not be norms on the full ambient space.

It was observed in \cite{JT26} that the extremal regime of the Gowers--Hatami uniformity functionals is classified by unitary-valued Leibman polynomial maps \cite{Leibman2002}. Specified to our context, a $1$-bounded function $F:A\to\mathrm{M}_N(\C)$ satisfies the extremal identity $\|F\|_{U^{k+1}(A)}=1$ if and only if $F$ is a unitary-valued Leibman polynomial map of degree at most $k$ . Consequently, the Weyl-orbit embedding gives
\[
U\in\CC^{(k)}\quad\Longleftrightarrow\quad F_U\text{ is a Leibman polynomial map of degree at most }k.
\]

The universal polynomial groups $\mathrm{Pol}_k(A)$ of the first author and Thom \cite{JT26} encode unital Leibman polynomial maps through group homomorphisms. Although $F_U(0)=U$ need not be the identity, its normalization $a\mapsto F_U(a)U^*$ is unital and obeys the same polynomial degree bound. Hence every unitary in the $k$th level of the Clifford hierarchy gives a unitary representation of $\mathrm{Pol}_k(A)$. We show in Section~\ref{sec:poly} that, up to scalar phase, these unitaries correspond precisely to the representations satisfying a natural compatibility relation with Weyl conjugation. This is a first systematic insight into the representation-theoretic structure of the Clifford hierarchy, reducing its classification to the study of the unitary representations of the $\mathrm{Pol}_k(A)$ groups.

In Section~\ref{sec:weyl-quotient-duality}, we analyze the structure of the quantum uniformity spaces inside the ambient matrix-valued spaces. We show that the range of the Weyl-orbit embedding is the fixed-point space of a natural uniformity preserving action. When the underlying spaces are equipped with their natural \(C^*\)-algebra structures, the Weyl-orbit embedding is an injective unital \(*\)-homomorphism, and the associated projection is a trace-preserving conditional expectation onto its range. We obtain a canonical decomposition of the ambient matrix-valued uniformity space into uniformity-preserving copies of the quantum uniformity space.
 
Bu, Gu, and Jaffe also formulated an interesting inverse conjecture for the quantum uniformity measures; see \cite[Conjecture~9]{BGJ25}. This conjecture can be interpreted as a Clifford-hierarchy testing problem. It asks, roughly, whether a unitary with large quantum uniformity measure must have nontrivial correlation, bounded below independently of \(n\), with a unitary in the corresponding level of the Clifford hierarchy. 

For unitary operators, the squared normalized Hilbert--Schmidt correlation is the fidelity between the normalized Choi states of the corresponding unitary channels; see, e.g., \cite[\S IV]{GLN05}. Thus the inverse problem may equivalently be viewed as asking whether large quantum uniformity measure forces nontrivial process fidelity with a gate in the corresponding level of the Clifford hierarchy. It is not difficult to prove a polynomial \(1\%\)-inverse theorem for the second quantum uniformity measure, see, e.g., \cite[\S 3]{BBCH26}.  
Hinsche et al.~\cite[Theorem 1.3]{HBvDEBH25} proved a polynomial \(1\%\)-inverse theorem for the third quantum uniformity measure in the qubit setting. Bao et al.~\cite{BBCH26} proved the \(99\%\)-inverse theorem for the fourth quantum uniformity measure in the arbitrary prime local dimension. Using the correspondence established here with the matrix-valued uniformity functionals, we hope to report progress on the inverse problem for the \(k\)th quantum uniformity measure for arbitrary $k$ in future work.

\subsection*{Acknowledgments}

Asgar Jamneshan was supported by the German Research Foundation under Germany's
Excellence Strategy -- EXC-2047 -- 390685813 and its Heisenberg Programme --
547294463. Berke Ünal was supported by the Koç University Presidential Fellowship. We are grateful to Jop Bri\"et for helpful discussions. 

\subsection*{AI tool disclosure}

All results in this paper were conceived and proved and the manuscript was written by the human authors. AI tools assisted with improving language.

\section{Background on Quantum and Matrix-Valued Uniformity Quantities}\label{sec:prelim}

We follow the conventions of \cite{BBCH26}.  Throughout, $p$ is a fixed prime number and $n\geq 1$ a fixed integer, $A:=\F_p^{2n}$ is a $2n$-dimensional vector space over $\F_p$, and $N:=p^n$. We write elements of $A$ as
$a=(u,v)$, with $u,v\in\F_p^n$.  Let
\[
    \omega:=e^{2\pi i/p},
    \qquad
    D_p:=
    \begin{cases}
        2p, & p=2,\\
        p, & p\neq2,
    \end{cases}
\]
and
\[
    \tau_p:=(-1)^p e^{\pi i/p}.
\]
Then $\tau_p^2=\omega$, and $\tau_p$ has order $D_p$.  On
$\C^{\F_p^n}$, define
\[
    X^v\ket{x}:=\ket{x+v},
    \qquad
    Z^u\ket{x}:=\omega^{x\cdot u}\ket{x}.
\]
For $a=(u,v)\in A$, define the Weyl operator
\[
    W_a
    :=
    \tau_p^{-\sum_{j=1}^n |u_jv_j|}Z^uX^v,
\]
where $|\cdot|:\F_p\to\{0,\ldots,p-1\}\subseteq\Z$ denotes the standard
section.

The symplectic form on $A$ is defined by 
\[
    [a,b]=[(u,v),(u',v')]:=u\cdot v'-u'\cdot v .
\]
There is a map $\beta:A\times A\to\Z/ D_p \Z$ such that
\begin{equation}\label{eq:twist}
    W_aW_b=\tau_p^{\beta(a,b)}W_{a+b},
    \qquad
    W_aW_b=\omega^{[a,b]}W_bW_a .
\end{equation}
The Pauli group is
\[
    \PG=\PG(n,p)
    :=
    \{\tau_p^tW_a:t\in\Z/D_p\Z,\ a\in A\}.
\]
Equivalently, $\PG$ is the image of the finite Heisenberg group $\Z/ D_p \Z\times A$ with multiplication
\[
    (t,a)(s,b):=(t+s+\beta(a,b),a+b)
\]
under the faithful representation
\[
    (t,a)\longmapsto \tau_p^tW_a .
\]
We set 
\[
\CC^{(1)}=\PG,
\]
and define for $k\geq 2$, the $k$th level of the Clifford hierarchy by 
\[
\CC^{(k)} :=\{U\in \U(N): U\CC^{(1)}U^*\subseteq \CC^{(k-1)}\}. 
\]
We introduce the quantum uniformity measures of Bu, Gu, and Jaffe\footnote{The same quantities are called Pauli uniformity norms in \cite{BBCH26}.}. 

On the space $\mathrm M_N(\C)$ of $N\times N$ complex matrices, we use the normalized
Hilbert--Schmidt inner product $\ip{S}{T}:=\frac1N\tr(S^*T)$ and the normalized Hilbert--Schmidt norm $\norm{T}_2^2:=\ip{T}{T}$. 
We denote by $\U(N)\subset \mathrm M_N(\C)$ the corresponding unitary group.

For $h\in A$ and $T\in \mathrm M_N(\C)$, define the derivative operator 
\[
    \partial_hT:=W_hTW_h^*T^*. 
\]
Iterated derivatives are defined by
\[
    \partial_{h_k,\ldots,h_1}
    :=
    \partial_{h_k}\circ\cdots\circ\partial_{h_1}.
\]
For $k\geq 1$, the $k$th quantum uniformity measure of $T$ is defined by the formula 
\[
    \norm{T}_{Q^k}^{2^k}
    :=
    \E_{h_1,\ldots,h_k\in A}
    \ip{\Id}{\partial_{h_k,\ldots,h_1}T}.
\]
Bu, Gu, and Jaffe established the following classification in \cite[Theorem 5]{BGJ25}.   
\begin{theorem}\label{thm:bgj}
    Let $k\geq 1$. For a unitary $U\in \U(N)$ it holds that 
\[
U\in \CC^{(k)} \qquad \text{if and only if} \qquad \norm{U}_{Q^{k+1}}=1.
\]
For \(k=1\), membership in \(\CC^{(1)}\) is understood modulo scalar phase, i.e. \(U\in\U(1)\CC^{(1)}\).
\end{theorem}
We define the matrix-valued uniformity functionals introduced by Gowers and Hatami in \cite{GH17}.  Let $G$ be a finite group. For a function $F:G\to \mathrm M_N(\C)$, the discrete derivative in direction $h\in G$ is defined by
\[
    (\Delta_h F)(x):=F(hx)F(x)^*.
\]
For $h_1,\ldots,h_k\in G$, put
\[
    \Delta_{h_k,\ldots,h_1}:=
    \Delta_{h_k}\circ\cdots\circ\Delta_{h_1}.
\]
For $k\geq 1$, the $k$th uniformity functional of
$F$ is defined by the formula 
\[
    \norm{F}_{U^k(G)}^{2^k}
    :=
    \E_{x,h_1,\ldots,h_k\in G}
    \ip{\Id}{(\Delta_{h_k,\ldots,h_1}F)(x)}.
\]
Define lists of cube words
\[
    c_k(h_1,\ldots,h_k)
    =
    \bigl(c_{1,k}(h_1,\ldots,h_k),\ldots,
          c_{2^k,k}(h_1,\ldots,h_k)\bigr)
\]
recursively by
\[
    c_0=(1),\qquad c_1(h_1)=(h_1,1),
\]
and
\[
    c_{k+1}(h_1,\ldots,h_{k+1})
    =
    \bigl(c_k(h_1,\ldots,h_k)h_{k+1},
          \overline{c_k(h_1,\ldots,h_k)}\bigr),
\]
where multiplication by \(h_{k+1}\) is applied to each entry of the list and
\(\overline{c_k}\) denotes the reversed list.

For functions
\(F_1,\ldots,F_{2^k}:G\to \mathrm M_N(\C)\), define the $k$th
Gowers multilinear form by the formula
\[
    \Lambda_{k,G}(F_1,\ldots,F_{2^k})
    :=
    \E_{x,h_1,\ldots,h_k\in G}
    \ip{\Id}{\prod_{i=1}^{2^k}
        F_i\bigl(c_{i,k}(h_1,\ldots,h_k)x\bigr)^{[i]}},
\]
where the product is ordered from left to right and
\[
    B^{[i]} :=
    \begin{cases}
        B, & i \ \mathrm{odd},\\
        B^*, & i \ \mathrm{even}.
    \end{cases}
\]
Then
\[
    \Lambda_{k,G}(F,\ldots,F)
    =
    \norm{F}_{U^k(G)}^{2^k}.
\]
When \(G=A\) is written additively, we write
\(c_{i,k}(h_1,\ldots,h_k)+x\) instead of
\(c_{i,k}(h_1,\ldots,h_k)x\). 

\section{Weyl-orbit embedding}\label{sec:main}
Define $\alpha\colon A\times \mathrm M_N(\C)\to \mathrm M_N(\C)$ by 
\[
    \alpha_a(S):=W_aSW_a^*,
    \qquad
    a\in A, S\in \mathrm M_N(\C).
\]
Since phases disappear under conjugation, it follows from \eqref{eq:twist} that
\begin{equation}
    \label{eq:action}
     \alpha_a\alpha_b=\alpha_{a+b}.
\end{equation}
Thus \(\alpha\) defines a trace-preserving action of the additive group
\(A\) on the \(C^*\)-algebra \(\mathrm M_N(\C)\) by \(*\)-automorphisms. 

 For $S\in \mathrm M_N(\C)$, we define its Weyl orbit map 
\begin{equation}\label{eq:weylorbit}
    F_S:A\to \mathrm M_N(\C),
    \qquad
    F_S(a):=\alpha_a(S).
\end{equation}
Using this map, we establish the following correspondence between the quantum uniformity measures and the matrix-valued uniformity functionals. 
\begin{theorem}
\label{thm:pullback}
For every $S\in \mathrm M_N(\C)$ and every $k\geq1$,
\[
    \norm{S}_{Q^k}
    =
    \norm{F_S}_{U^k(A)}.
\]
\end{theorem}

\begin{proof}
Define
\[
    \mathcal T:\mathrm M_N(\C)\to \mathrm M_N(\C)^A,
    \qquad
    \mathcal T(S)(x):=\alpha_x(S).
\]
Then $F_S=\mathcal T(S)$.  For $h,x\in A$ and $S\in \mathrm M_N(\C)$,
\[
\begin{aligned}
    (\Delta_h \mathcal T(S))(x)
    &=
    \mathcal T(S)(h+x)\mathcal T(S)(x)^*                                      \\
    &=
    \alpha_{h+x}(S)\alpha_x(S)^*                            \\
    &=
    \alpha_x(\alpha_h(S)S^*)                                \\
    &=
    \mathcal T(\partial_hS)(x).
\end{aligned}
\]
Thus
\[
    \Delta_h\mathcal T(S)=\mathcal T(\partial_hS).
\]
Iterating gives
\[
    (\Delta_{h_k,\ldots,h_1}F_S)(x)
    =
    \alpha_x\!\left(\partial_{h_k,\ldots,h_1}S\right).
\]
Taking normalized traces removes the outer conjugation:
\[
    \ip{\Id}{(\Delta_{h_k,\ldots,h_1}F_S)(x)}
    =
    \ip{\Id}{\partial_{h_k,\ldots,h_1}S}.
\]
Averaging over $x,h_1,\ldots,h_k$ proves the identity.
\end{proof}

We first make explicit the relation between the quantum multilinear
functional of Bu, Gu, and Jaffe and the matrix-valued Gowers multilinear forms defined in Section \ref{sec:prelim}. 

Define an ordering
\[
    \sigma^{(k)}
    =
    \bigl(\sigma^{(k)}_1,\ldots,\sigma^{(k)}_{2^k}\bigr)
\]
of \(\{0,1\}^k\) recursively by
\[
    \sigma^{(0)}=(\varnothing),
    \qquad
    \sigma^{(1)}=((0),(1)),
\]
and
\[
    \sigma^{(j+1)}
    =
    \bigl(
        (0,\sigma^{(j)}_1),\ldots,(0,\sigma^{(j)}_{2^j}),
        (1,\sigma^{(j)}_{2^j}),\ldots,(1,\sigma^{(j)}_1)
    \bigr).
\]

Let $(T_\varepsilon)_{\varepsilon\in\{0,1\}^k}
    \subseteq \mathrm M_N(\C)$. 
For \(j\geq 0\), define \(\mathcal D_j\) recursively by
\[
    \mathcal D_0(T_\varnothing):=T_\varnothing
\]
and
\[
\begin{aligned}
    &\mathcal D_{j+1}
    \bigl((T_\varepsilon)_{\varepsilon\in\{0,1\}^{j+1}};
    h_{j+1},\ldots,h_1\bigr)                                      \\
    &\qquad :=
    \alpha_{h_{j+1}}\!\left(
    \mathcal D_j
    \bigl((T_{(0,\eta)})_{\eta\in\{0,1\}^j};
    h_j,\ldots,h_1\bigr)\right)                                    \\
    &\qquad\quad \cdot
    \mathcal D_j
    \bigl((T_{(1,\eta)})_{\eta\in\{0,1\}^j};
    h_j,\ldots,h_1\bigr)^* .
\end{aligned}
\]
Then the quantum multilinear form introduced by Bu, Gu, and Jaffe in \cite[Definition 35]{BGJ25} is the formula 
\[
    \Lambda_k^Q
    \bigl((T_\varepsilon)_{\varepsilon\in\{0,1\}^k}\bigr)
    :=
    \E_{h_1,\ldots,h_k\in A}
    \ip{\Id}{
    \mathcal D_k
    \bigl((T_\varepsilon)_{\varepsilon\in\{0,1\}^k};
    h_k,\ldots,h_1\bigr)} .
\]

\begin{lemma}\label{lem:multlin}
For every \(k\geq 1\) and every family $(S_\varepsilon)_{\varepsilon\in\{0,1\}^k}
    \subseteq \mathrm M_N(\C)$, we have 
\[
    \Lambda_k^Q
    \bigl((S_\varepsilon)_{\varepsilon\in\{0,1\}^k}\bigr)
    =
    \Lambda_{k,A}
    \bigl(
        F_{S_{\sigma^{(k)}_1}},
        \ldots,
        F_{S_{\sigma^{(k)}_{2^k}}}
    \bigr), 
\]
where $F_S$ denotes the Weyl orbit map \eqref{eq:weylorbit}. 
\end{lemma}

\begin{proof}
We first introduce an auxiliary recursive multilinear expression.  For
matrix-valued functions
\[
    (F_\varepsilon)_{\varepsilon\in\{0,1\}^j},
    \qquad
    F_\varepsilon:A\to \mathrm M_N(\C),
\]
set
\[
    \mathcal Q_0(F_\varnothing;x):=F_\varnothing(x),
\]
and define
\[
\begin{aligned}
    &\mathcal Q_{j+1}
    \bigl((F_\varepsilon)_{\varepsilon\in\{0,1\}^{j+1}};
    x,h_1,\ldots,h_{j+1}\bigr)                                    \\
    &\qquad :=
    \mathcal Q_j
    \bigl((F_{(0,\eta)})_{\eta\in\{0,1\}^j};
    h_{j+1}+x,h_1,\ldots,h_j\bigr)                                  \\
    &\qquad\quad \cdot
    \mathcal Q_j
    \bigl((F_{(1,\eta)})_{\eta\in\{0,1\}^j};
    x,h_1,\ldots,h_j\bigr)^* .
\end{aligned}
\]

We claim first that
\begin{equation}\label{eq:Q-product-expansion}
    \mathcal Q_j
    \bigl((F_\varepsilon)_{\varepsilon\in\{0,1\}^j};
    x,h_1,\ldots,h_j\bigr)
    =
    \prod_{i=1}^{2^j}
    F_{\sigma_i^{(j)}}
    \bigl(c_{i,j}(h_1,\ldots,h_j)+x\bigr)^{[i]} .
\end{equation}
This is immediate for \(j=0\).  Assume it holds at level \(j\), and put
\(m=2^j\).  By the recursive definition of \(\mathcal Q_{j+1}\), the first
block is
\[
    \prod_{i=1}^{m}
    F_{(0,\sigma_i^{(j)})}
    \bigl(c_{i,j}(h_1,\ldots,h_j)+h_{j+1}+x\bigr)^{[i]} .
\]
The second block, after taking the adjoint, becomes
\[
    \prod_{r=1}^{m}
    F_{(1,\sigma_{m+1-r}^{(j)})}
    \bigl(c_{m+1-r,j}(h_1,\ldots,h_j)+x\bigr)^{[m+r]} .
\]
Here the adjoint reverses the order of the product and flips the parity of
the exponent \([\,\cdot\,]\).  Hence the recursive definitions of
\(c_{j+1}\) and \(\sigma^{(j+1)}\) give exactly
\[
    \mathcal Q_{j+1}
    \bigl((F_\varepsilon)_\varepsilon;
    x,h_1,\ldots,h_{j+1}\bigr)
    =
    \prod_{i=1}^{2^{j+1}}
    F_{\sigma_i^{(j+1)}}
    \bigl(c_{i,j+1}(h_1,\ldots,h_{j+1})+x\bigr)^{[i]} .
\]
This proves \eqref{eq:Q-product-expansion} for all \(j\).

We next prove the Weyl-orbit identity
\begin{equation}\label{eq:Q-D-intertwining}
    \mathcal Q_j
    \bigl((F_{T_\varepsilon})_{\varepsilon\in\{0,1\}^j};
    x,h_1,\ldots,h_j\bigr)
    =
    \alpha_x\!\left(
    \mathcal D_j
    \bigl((T_\varepsilon)_{\varepsilon\in\{0,1\}^j};
    h_j,\ldots,h_1\bigr)\right). 
\end{equation}
For \(j=0\), this is immediate:
\[
    \mathcal Q_0(F_{T_\varnothing};x)
    =
    F_{T_\varnothing}(x)
    =
    \alpha_x(T_\varnothing)
    =
    \alpha_x(\mathcal D_0(T_\varnothing)).
\]

Assume \eqref{eq:Q-D-intertwining} holds at level \(j\).  Put
\[
    D_0
    :=
    \mathcal D_j
    \bigl((T_{(0,\eta)})_{\eta\in\{0,1\}^j};
    h_j,\ldots,h_1\bigr),
    \qquad
    D_1
    :=
    \mathcal D_j
    \bigl((T_{(1,\eta)})_{\eta\in\{0,1\}^j};
    h_j,\ldots,h_1\bigr).
\]
Then
\[
\begin{aligned}
    &\mathcal Q_{j+1}
    \bigl((F_{T_\varepsilon})_{\varepsilon\in\{0,1\}^{j+1}};
    x,h_1,\ldots,h_{j+1}\bigr)                                    \\
    &\qquad =
    \mathcal Q_j
    \bigl((F_{T_{(0,\eta)}})_\eta;
    h_{j+1}+x,h_1,\ldots,h_j\bigr)
    \mathcal Q_j
    \bigl((F_{T_{(1,\eta)}})_\eta;
    x,h_1,\ldots,h_j\bigr)^*                                      \\
    &\qquad =
    \alpha_{h_{j+1}+x}(D_0)\,\alpha_x(D_1)^* .
\end{aligned}
\]
Since \(\alpha\) is an action by \(*\)-automorphisms,
\[
\begin{aligned}
    \alpha_{h_{j+1}+x}(D_0)\,\alpha_x(D_1)^*
    &=
    \alpha_x(\alpha_{h_{j+1}}(D_0))\,\alpha_x(D_1^*)                 \\
    &=
    \alpha_x\!\left(\alpha_{h_{j+1}}(D_0)D_1^*\right)                \\
    &=
    \alpha_x\!\left(
    \mathcal D_{j+1}
    \bigl((T_\varepsilon)_{\varepsilon\in\{0,1\}^{j+1}};
    h_{j+1},h_j,\ldots,h_1\bigr)\right).
\end{aligned}
\]
This proves \eqref{eq:Q-D-intertwining} for all \(j\).

Finally, applying \eqref{eq:Q-product-expansion} with
\(F_\varepsilon=F_{T_\varepsilon}\), we get
\[
\begin{aligned}
    &\Lambda_{k,A}
    \bigl(
        F_{T_{\sigma^{(k)}_1}},
        \ldots,
        F_{T_{\sigma^{(k)}_{2^k}}}
    \bigr)                                                        \\
    &\qquad =
    \E_{x,h_1,\ldots,h_k\in A}
    \ip{\Id}{
    \mathcal Q_k
    \bigl((F_{T_\varepsilon})_{\varepsilon\in\{0,1\}^k};
    x,h_1,\ldots,h_k\bigr)}                                      \\
    &\qquad =
    \E_{x,h_1,\ldots,h_k\in A}
    \ip{\Id}{
    \alpha_x\!\left(
    \mathcal D_k
    \bigl((T_\varepsilon)_{\varepsilon\in\{0,1\}^k};
    h_k,\ldots,h_1\bigr)\right)} .
\end{aligned}
\]
Since $\alpha$ is trace-preserving, 
\[
\begin{aligned}
    \Lambda_{k,A}
    \bigl(
        F_{T_{\sigma^{(k)}_1}},
        \ldots,
        F_{T_{\sigma^{(k)}_{2^k}}}
    \bigr)
    &=
    \E_{h_1,\ldots,h_k\in A}
    \ip{\Id}{
    \mathcal D_k
    \bigl((T_\varepsilon)_{\varepsilon\in\{0,1\}^k};
    h_k,\ldots,h_1\bigr)}                                      \\
    &=
    \Lambda_k^Q
    \bigl((T_\varepsilon)_{\varepsilon\in\{0,1\}^k}\bigr).
\end{aligned}
\]
This proves the desired identity.
\end{proof}

\section{Leibman polynomial maps and universal polynomial groups}
\label{sec:poly}

We recall Leibman's definition of polynomial maps between
groups \cite{Leibman2002}.  A map \(\phi:G\to H\) is polynomial of degree at most
\(-1\) if it is identically \(1_H\).  For \(k\geq0\), it is polynomial of degree at
most \(k\) if, for every \(h\in G\), the discrete derivative 
\[
    (\Delta_h\phi)(g):=\phi(hg)\phi(g)^{-1}
\]
is polynomial of degree at most \(k-1\). 

We have the following equivalence between the extremizers of quantum uniformity measures and matrix-valued uniformity functional. 

\begin{proposition}
\label{prop:exact-extremality}
Let \(U\in\U(N)\) and let \(k\geq1\). The following are equivalent:
\begin{enumerate}[label=\textup{(\roman*)}]
    \item \(\norm{U}_{Q^{k+1}}=1\);
    \item the Weyl orbit map \(F_U:A\to\U(N)\) is Leibman polynomial of degree at most \(k\).
\end{enumerate}
Consequently, these conditions are equivalent to
\(U\in\CC^{(k)}\), where for \(k=1\), membership in \(\CC^{(1)}\) is understood modulo scalar phase, i.e. \(U\in\U(1)\CC^{(1)}\).
\end{proposition}

\begin{proof}
By Theorem~\ref{thm:pullback}, we have
\[
    \norm{U}_{Q^{k+1}}
    =
    \norm{F_U}_{U^{k+1}(A)} .
\]
Since \(F_U\) is unitary-valued, every \((k+1)\)-fold discrete derivative
\[
    \Delta_{h_{k+1}}\cdots\Delta_{h_1}F_U(x)
\]
is again unitary. Therefore, for every \(x,h_1,\ldots,h_{k+1}\in A\),
\[
    \ip{\Id}{\Delta_{h_{k+1}}\cdots\Delta_{h_1}F_U(x)}
\]
lies in the closed unit disk. Moreover, if \(V\in\U(N)\), then
\(\ip{\Id}{V}=1\) if and only if \(V=\Id\).

Since the average of finitely many points in the closed unit disk is equal
to \(1\) only when all points are equal to \(1\), the defining average for
\(\norm{F_U}_{U^{k+1}(A)}^{2^{k+1}}\) is equal to \(1\) if and only if
\[
    \Delta_{h_{k+1}}\cdots\Delta_{h_1}F_U(x)=\Id
\]
for every \(x,h_1,\ldots,h_{k+1}\in A\). By Leibman's definition, this is
equivalent to \(F_U\) being polynomial of degree at most \(k\).
This proves the equivalence of \textup{(i)} and \textup{(ii)}.

The final assertions follow from Theorem~\ref{thm:bgj}.
\end{proof}

We recall the universal construction of Jamneshan and Thom \cite[Theorem~1.2 and Remark~2.5]{JT26}. Let $G$ be a group and $k\geq1$. Let $\mathsf F_G$ be the free group on symbols $t_g$, $g\in G$, and let $t:G\to\mathsf F_G$ be the tautological map. Define
\[
\mathrm{Pol}_k(G):=\mathsf F_G\Big/\Big\langle\!\Big\langle t_{1_G},\; (\Delta_{h_{k+1}}\cdots\Delta_{h_1}t)(g):g,h_1,\ldots,h_{k+1}\in G\Big\rangle\!\Big\rangle,
\]
where the double brackets denote normal closure. The map $\iota_k:G\to\mathrm{Pol}_k(G)$ induced by $t$ is a unital Leibman polynomial map of degree at most $k$. By the defining relations, composition with $\iota_k$ gives a natural bijection
\[
\operatorname{Hom}(\mathrm{Pol}_k(G),H)\;\cong\;\{\phi:G\to H:\ \phi(1_G)=1_H,\ \deg\phi\leq k\}
\]
for every group $H$. Thus $\mathrm{Pol}_k(G)$ classifies unital polynomial maps of degree at most $k$.

There is no restriction to unital maps in applications: if $f:G\to H$ has degree at most $k$ and $c=f(1_G)$, then $\phi(g):=f(g)c^{-1}$ is unital, and $\Delta_h\phi=\Delta_h f$ for every $h\in G$. Hence $\phi$ also has degree at most $k$, and there is a unique homomorphism $\pi:\mathrm{Pol}_k(G)\to H$ such that $f(g)=\pi(\iota_k(g))c$.

\begin{theorem}\label{prop:universal-clifford}
Let $k\geq1$ and $U\in\U(N)$. Then $U\in\CC^{(k)}$ (with membership understood up to scalar phase when $k=1$) if and only if there is a unitary representation $\pi_U:\mathrm{Pol}_k(A)\to\U(N)$ satisfying
\[
\pi_U(\iota_k(a))=F_U(a)U^*=W_aUW_a^*U^*,\qquad a\in A.
\]
This representation is unique and determines $U$ up to scalar phase. Moreover, a unitary representation $\pi:\mathrm{Pol}_k(A)\to\U(N)$ arises in this way precisely when $D_a:=\pi(\iota_k(a))$ satisfies
\begin{equation}\label{eq:weyl-compatibility}
D_{a+b}=\alpha_a(D_b)D_a\qquad(a,b\in A).
\end{equation}
\end{theorem}

\begin{proof}
By Proposition~\ref{prop:exact-extremality}, $U$ has the stated Clifford-hierarchy membership exactly when $F_U$ is polynomial of degree at most $k$. The map $\phi_U(a):=F_U(a)U^*$ is unital, and $\Delta_h\phi_U=\Delta_hF_U$ for all $h\in A$. The universal property therefore gives the desired equivalence and uniqueness of $\pi_U$.

If $\pi_U=\pi_{U'}$, then $\alpha_a(U)U^*=\alpha_a(U')U'^*$ for all $a\in A$, so $\alpha_a(U'^*U)=U'^*U$. Since the Weyl operators span $\mathrm M_N(\C)$, their commutant consists of the scalars. Thus $U'$ differs from $U$ by a scalar phase.

The identity \eqref{eq:weyl-compatibility} follows directly from $\phi_U(a)=\alpha_a(U)U^*$ and the group-action property of $\alpha$. Conversely, suppose that \eqref{eq:weyl-compatibility} holds and set $V_a:=D_a^*W_a$. Then
\[
V_aV_b=\tau_p^{\beta(a,b)}V_{a+b}\qquad(a,b\in A),
\]
so the $V_a$ obey the same Weyl multiplication rules as the $W_a$. The nondegenerate Weyl commutation relations imply that the $V_a$ form an orthogonal basis of $\mathrm M_N(\C)$. Consequently, $W_a\mapsto V_a$ extends to a unital $*$-automorphism of $\mathrm M_N(\C)$, which is inner. Thus $V_a=UW_aU^*$ for some $U\in\U(N)$, and hence $D_a=F_U(a)U^*$. Since $D=\pi\circ\iota_k$ is polynomial of degree at most $k$, the first part proves the claimed membership of $U$ in the Clifford hierarchy.
\end{proof}

\section{Structural Properties of the Weyl-orbit Embedding}
\label{sec:weyl-quotient-duality}
\subsection{Uniformity spaces and the Weyl-orbit embedding $\mathcal{T}$} For the purposes of this section, we package together the underlying
\(C^*\)-algebra, its normalized trace, and the family of uniformity functionals,
and refer to the resulting structure as a \emph{uniformity space}.

The \emph{quantum uniformity space} is
\[
\mathsf Q:=\mathrm M_N(\mathbb C),
\]
equipped with its usual \(C^*\)-algebra structure, normalized trace
\(
\tau(S):=\frac{1}{N}\operatorname{tr}(S),
\)
and, for each \(k\geq2\), the quantum uniformity measure
\(\|\cdot\|_{Q^k}\).

The \emph{matrix-valued uniformity space} is
\[
\mathsf G:=\mathrm M_N(\mathbb C)^A,
\]
equipped with its pointwise \(C^*\)-algebra structure, the \(C^*\)-norm
\(
\|F\|_\infty:=\max_{x\in A}\|F(x)\|_{\mathrm{op}},
\)
the normalized trace
\(
\tau_A(F):=\mathbb E_{x\in A}\tau(F(x)),
\)
and, for each \(k\geq2\), the matrix-valued uniformity functional
\(\|\cdot\|_{U^k(A)}\).

The \emph{Weyl-orbit embedding} is the linear map
\begin{equation}
\label{eq:T-definition-short}
\mathcal T:\mathsf Q\longrightarrow\mathsf G,
\qquad
S\longmapsto F_S.
\end{equation}
It is an embedding of the quantum uniformity space into the matrix-valued
uniformity space in the following sense:

\begin{proposition}
\label{prop:Weyl-orbit-embedding}
The Weyl-orbit embedding
\(
\mathcal T
\)
is an embedding of uniformity spaces. More precisely:
\begin{enumerate}[label=\textup{(\roman*)}]
\item For every \(k\geq2\),
\[
\|\mathcal T(S)\|_{U^k(A)}
=
\|S\|_{Q^k},
\qquad S\in\mathsf Q.
\]

\item The map \(\mathcal T\) is an injective unital trace-preserving
\(*\)-homomorphism.
\end{enumerate}
\end{proposition}

\begin{proof}
Part \textup{(i)} is exactly Theorem~\ref{thm:pullback}. To prove
\textup{(ii)}, first note that \(\mathcal T\) is injective by \textup{(i)}.
Moreover, for each \(x\in A\), since \(W_x\) is unitary, the map
\[
\alpha_x:\mathsf Q\longrightarrow\mathsf Q,
\qquad
\alpha_x(S)=W_xSW_x^*,
\]
is a unital \(*\)-automorphism. Since \(\mathcal T(S)(x)=\alpha_x(S)\) and the algebraic operations in
\(\mathsf G\) are pointwise, we have
\[
\mathcal T(S_1S_2)=\mathcal T(S_1)\mathcal T(S_2),\qquad
\mathcal T(S^*)=\mathcal T(S)^*,\qquad
\mathcal T(I)=I.
\]
\[
\tau_A(\mathcal T(S))
=
\mathbb E_{x\in A}\tau(\alpha_x(S))
=
\mathbb E_{x\in A}\tau(S)
=
\tau(S).
\]

Thus \(\mathcal T\) is an injective unital trace-preserving \(*\)-homomorphism.\end{proof}

We denote the range of the Weyl-orbit embedding by
\begin{equation}
\label{eq:Q-tilde-definition}
\widetilde{\mathsf Q}
:=
\operatorname{ran}\mathcal T
\subseteq\mathsf G,
\end{equation}
and refer to \(\widetilde{\mathsf Q}\) as the \emph{embedded quantum uniformity space}.

\subsection{Action $\rho$}
We next introduce a natural symmetry of the ambient matrix-valued uniformity space \(\mathsf G\).
Define an action
\begin{equation}
\label{eq:rho-definition-short}
\rho:A\curvearrowright\mathsf G,
\qquad
\bigl(\rho_b(F)\bigr)(x)
:=
\alpha_b\bigl(F(x-b)\bigr).
\end{equation}

\begin{proposition}
\label{prop:rho-structure}
The action \(\rho\) preserves the uniformity-space structure of \(\mathsf G\).
More precisely, for every \(b\in A\):
\begin{enumerate}[label=\textup{(\roman*)}]
\item \(\rho_b\) is uniformity preserving for every matrix-valued uniformity functional,
\[
\|\rho_b(F)\|_{U^k(A)}
=
\|F\|_{U^k(A)},
\qquad
k\geq2;
\]
\item \(\rho_b\) is a \(\tau_A\)-preserving \(*\)-automorphism of the
\(C^*\)-algebra \(\mathsf G\).
\end{enumerate}
\end{proposition}

\begin{proof}
For \(h,x\in A\), multiplicativity of \(\alpha_b\) gives
\[
\bigl(\Delta_h(\rho_bF)\bigr)(x)
=
\alpha_b\bigl((\Delta_hF)(x-b)\bigr).
\]
Iterating,
\[
\bigl(\Delta_{h_k}\cdots\Delta_{h_1}(\rho_bF)\bigr)(x)
=
\alpha_b\!\left(
\bigl(\Delta_{h_k}\cdots\Delta_{h_1}F\bigr)(x-b)
\right).
\]
Since \(\alpha_b\) preserves the normalized trace, and averaging over \(x\)
is translation invariant, this proves \textup{(i)}.

For \textup{(ii)}, since \(\alpha_b\) is a \(*\)-automorphism and the
operations on \(\mathsf G\) are pointwise,
\[
\rho_b(FG)=\rho_b(F)\rho_b(G),
\qquad
\rho_b(F^*)=\rho_b(F)^*,
\qquad
\rho_b(I)=I.
\]
Moreover,
\[
\rho_b\circ\rho_{-b}
=
\rho_{-b}\circ\rho_b
=
\operatorname{id}_{\mathsf G},
\]
so \(\rho_b\) is a \(*\)-automorphism. Finally,
\[
\tau_A(\rho_b(F))
=
\mathbb E_{x\in A}
\operatorname{\tau}\left(\alpha_b(F(x-b))\right)
=
\tau_A(F),
\]
which proves that \(\rho_b\) is trace preserving.
\end{proof}

\begin{proposition}
\label{prop:fixed-point-range-short}
Embedded quantum uniformity space $\widetilde{\mathsf Q}$ is the fixed-point space of the action
\(\rho\):
\begin{equation}
\label{eq:fixed-point-range-short}
 \widetilde{\mathsf Q}
 =
 \mathsf G^\rho
 :=
 \left\{
 F\in\mathsf G:
 \rho_b(F)=F
 \text{ for every }b\in A
 \right\}.
\end{equation}
\end{proposition}

\begin{proof}
Let \(F_S=\mathcal T(S)\).  Then, for every \(b,x\in A\),
\[
 \bigl(\rho_b(F_S)\bigr)(x)
 =
 \alpha_b\bigl(F_S(x-b)\bigr)\\
 =
 \alpha_b\bigl(\alpha_{x-b}(S)\bigr)\\
 =
 \alpha_x(S)\\
 =
 F_S(x).
\]
Thus every element of \(\widetilde{\mathsf Q}\) is fixed by every
\(\rho_b\).

Conversely, suppose that
\(
 \rho_b(F)=F
\)
for every $b\in A$.
Taking \(b=x\), we obtain
\[
 F(x)
 =
 \bigl(\rho_x(F)\bigr)(x)
 =
 \alpha_x(F(0))
 =
 \bigl(\mathcal T(F(0))\bigr)(x).
\]
Hence
\[
 F=\mathcal T(F(0))
 \in
 \widetilde{\mathsf Q}.
\]
\end{proof}
\subsection{Retract $\mathcal R$ and Projection $\mathcal P$}
We next construct the canonical retraction onto the quantum uniformity space.
Define
\begin{equation}
\label{eq:R-definition-short}
\mathcal R:\mathsf G\longrightarrow\mathsf Q,
\qquad
\mathcal R(F)
:=
\mathbb E_{x\in A}\alpha_{-x}\bigl(F(x)\bigr).
\end{equation}
The associated projection onto the embedded quantum uniformity space is
\begin{equation}
\label{eq:P-definition-short}
\mathcal P
:=
\mathcal T\mathcal R:
\mathsf G\longrightarrow\mathsf G.
\end{equation}

\begin{lemma}
\label{lem:R-P-basic}
The maps \(\mathcal R\) and \(\mathcal P\) satisfy the following properties.
\begin{enumerate}[label=\textup{(\roman*)}]
\item The map \(\mathcal R\) is a left inverse of \(\mathcal T\):
\[
\mathcal R\mathcal T
=
\operatorname{id}_{\mathsf Q}.
\]
In particular, \(\mathcal R\) is surjective and \(\mathcal P\) is a
projection with
\[
\operatorname{ran}\mathcal P
=
\widetilde{\mathsf Q},
\qquad
\ker\mathcal P
=
\ker\mathcal R.
\]

\item The projection \(\mathcal P\) is the averaging operator associated
with the action \(\rho\):
\begin{equation}
\label{eq:P-average-short}
\mathcal P
=
\mathbb E_{b\in A}\rho_b.
\end{equation}
\end{enumerate}
\end{lemma}

\begin{proof}
For \(S\in\mathsf Q\),
\[
\mathcal R(\mathcal T(S))
=
\mathbb E_{x\in A}
\alpha_{-x}\bigl(\alpha_x(S)\bigr)
=
S,
\]
which proves the first assertion in \textup{(i)}. Hence
\(\mathcal P^2=\mathcal P\). Moreover,
\[
\operatorname{ran}\mathcal P
=
\operatorname{ran}\mathcal T
=
\widetilde{\mathsf Q},
\]
and, since \(\mathcal T\) is injective,
\[
\mathcal P(F)=0
\iff
\mathcal T(\mathcal R(F))=0
\iff
\mathcal R(F)=0.
\]
Thus \(\ker\mathcal P=\ker\mathcal R\).

For \textup{(ii)}, for every \(x\in A\), the substitution \(b=x-y\) gives
\begin{align*}
\bigl(\mathcal P(F)\bigr)(x)
&=
\alpha_x\bigl(\mathcal R(F)\bigr)\\
&=
\mathbb E_{y\in A}
\alpha_{x-y}\bigl(F(y)\bigr)\\
&=
\mathbb E_{b\in A}
\alpha_b\bigl(F(x-b)\bigr)\\
&=
\mathbb E_{b\in A}
\bigl(\rho_b(F)\bigr)(x).
\end{align*}
\end{proof}

Recall that if \(\mathcal A\) is a unital \(C^*\)-algebra and \(\mathcal B\subseteq\mathcal A\) is a unital \(C^*\)-subalgebra, a conditional expectation from \(\mathcal A\) onto \(\mathcal B\) is a unital completely positive projection
\[
 E:\mathcal A\longrightarrow\mathcal B
\]
satisfying the \(\mathcal B\)-bimodule identity
\begin{equation}
\label{eq:abstract-conditional-expectation}
 E(B_1XB_2)
 =
 B_1E(X)B_2
 \qquad
 \bigl(
 X\in\mathcal A,\;
 B_1,B_2\in\mathcal B
 \bigr).
\end{equation}

\begin{proposition}
\label{prop:P-conditional-expectation-short}
The map
$ \mathcal P$
is a \(\tau_A\)-preserving conditional expectation onto the embedded quantum
uniformity space \(\widetilde{\mathsf Q}\).
\end{proposition}

\begin{proof}
By Proposition~\ref{prop:Weyl-orbit-embedding}, \(\widetilde{\mathsf Q}\) is a unital \(C^*\)-subalgebra of \(\mathsf G\). By Lemma~\ref{lem:R-P-basic},
\begin{equation}
\label{eq:P-average-Cstar-short}
 \mathcal P
 =
 \mathbb E_{b\in A}\rho_b.
\end{equation}
By Proposition~\ref{prop:rho-structure}, each \(\rho_b\) is a unital \(*\)-automorphism, and hence is completely positive. Consequently, their average \(\mathcal P\) is unital and completely positive.

By Lemma~\ref{lem:R-P-basic}\textup{(i)},
\(
 \operatorname{ran}\mathcal P
 =
 \widetilde{\mathsf Q}
\)
and \(\mathcal P^2=\mathcal P\), thus the map \(\mathcal P\) is a unital completely positive projection from \(\mathsf G\) onto
\(\widetilde{\mathsf Q}\).

Let
\(
 F\in\mathsf G,\,
 B_1,B_2\in\widetilde{\mathsf Q}.
\)
By the fixed-point characterization given by
Proposition~\ref{prop:fixed-point-range-short},
\[
 \rho_b(B_1)=B_1,
 \qquad
 \rho_b(B_2)=B_2
\]
for every \(b\in A\). Therefore
\begin{align*}
 \mathcal P(B_1FB_2)
 &=
 \mathbb E_{b\in A}\rho_b(B_1FB_2)\\
 &=
 \mathbb E_{b\in A}
 \rho_b(B_1)\rho_b(F)\rho_b(B_2)\\
 &=
 B_1
 \left(
 \mathbb E_{b\in A}\rho_b(F)
 \right)
 B_2\\
 &=
 B_1\mathcal P(F)B_2.
\end{align*}
This proves the bimodule identity.

Finally, Proposition~\ref{prop:rho-structure} gives
\[
 \tau_A(\mathcal P(F))
 =
 \mathbb E_{b\in A}\tau_A(\rho_b(F))
 =
 \tau_A(F).
\]
Thus \(\mathcal P\) is a
\(\tau_A\)-preserving conditional expectation from
\(\mathsf G\) onto
\(\widetilde{\mathsf Q}\).
\end{proof}

In quantum-information language, the algebra
\(
    \mathsf G
    \cong \ell_\infty(A)\otimes \mathrm M_N(\C)
\)
is the block-diagonal observable algebra of a classical-quantum system
with classical register \(A\). The action \(\rho\) combines translations
of the classical register with Weyl conjugations of the quantum register.
Averaging such a symmetry action is the standard twirling construction.
In the present setting, we showed that the canonical projection arising from the uniformity-norm geometry is precisely this twirling map. Proposition~\ref{prop:P-conditional-expectation-short}
therefore identifies $\widetilde{\mathsf Q}$ with the invariant observable algebra and \(\mathcal P\) with the corresponding trace-preserving conditional
expectation.

\subsection{Adjoint relations}
We use the normalized Hilbert--Schmidt pairings
\[
\langle S,T\rangle:=\tau(S^*T),
\qquad
\langle F,G\rangle_A:=\tau_A(F^*G)
=
\mathbb E_{x\in A}\tau\bigl(F(x)^*G(x)\bigr)
\]
on \(\mathsf Q\) and \(\mathsf G\), respectively.

\begin{proposition}
\label{prop:adjoint-relations-short}
For every \(S\in\mathsf Q\) and
\(F\in\mathsf G\),
\begin{equation}
\label{eq:adjoint-identity-short}
 \langle\mathcal T(S),F\rangle_A
 =
 \langle S,\mathcal R(F)\rangle.
\end{equation}
Thus, with respect to the Hilbert--Schmidt pairings,
\begin{equation}
\label{eq:adjoint-operators-short}
 \mathcal T^*
 =
 \mathcal R,
 \qquad
 \mathcal P^*
 =
 \mathcal P.
\end{equation}
Consequently,
\begin{equation}
\label{eq:orthogonal-complement-short}
 \ker\mathcal R
 =
 \bigl(\widetilde{\mathsf Q}\bigr)^\perp,
\end{equation}
and \(\mathcal P\) is the Hilbert--Schmidt orthogonal projection onto
\(\widetilde{\mathsf Q}\).
\end{proposition}

\begin{proof}
For \(S,T\in\mathsf Q\), since
\(\alpha_x(S)=W_xSW_x^*\) and \(W_x\) is unitary, we have
\[
\langle \alpha_x(S),T\rangle
=
\langle S,\alpha_{-x}(T)\rangle.
\]
Hence, for \(S\in\mathsf Q\) and \(F\in\mathsf G\),
\[
\begin{aligned}
\langle \mathcal T(S),F\rangle_A
&=
\mathbb E_{x\in A}
\langle \alpha_x(S),F(x)\rangle\\
&=
\mathbb E_{x\in A}
\langle S,\alpha_{-x}(F(x))\rangle\\
&=
\langle S,\mathcal R(F)\rangle.
\end{aligned}
\]
Hence
\[
\mathcal T^*=\mathcal R.
\]
Taking adjoints, \(\mathcal R^*=\mathcal T\), and therefore
\[
\mathcal P^*
=
(\mathcal T\mathcal R)^*
=
\mathcal R^*\mathcal T^*
=
\mathcal T\mathcal R
=
\mathcal P.
\]
Moreover,
\[
\ker\mathcal R
=
\ker\mathcal T^*
=
(\operatorname{ran}\mathcal T)^\perp
=
(\widetilde{\mathsf Q})^\perp.
\]
Since \(\mathcal P^2=\mathcal P\), \(\mathcal P^*=\mathcal P\), and
\(\operatorname{ran}\mathcal P=\widetilde{\mathsf Q}\), it follows that
\(\mathcal P\) is the Hilbert--Schmidt orthogonal projection onto
\(\widetilde{\mathsf Q}\).
\end{proof}

\subsection{Quantum Decomposition of the Ambient Uniformity Space}

We now decompose the ambient uniformity space into copies of the quantum uniformity space.

For each character \(\chi\in\widehat A\), define the
\emph{\(\chi\)-twisted Weyl-orbit embedding}
\[
\mathcal T_\chi:\mathsf Q\longrightarrow\mathsf G,
\qquad
\mathcal T_\chi(S)(x)
:=
\overline{\chi(x)}\,\alpha_x(S).
\]
For the trivial character \(1\in\widehat A\), the map
\(\mathcal T_1\) is the Weyl-orbit embedding \(\mathcal T\). Define also
\[
\mathcal R_\chi(F):=
\mathbb E_{x\in A}\chi(x)\alpha_{-x}(F(x)),
\qquad
\mathcal P_\chi:=\mathcal T_\chi\mathcal R_\chi.
\]
They satisfy the natural twisted analogue of
Lemma~\ref{lem:R-P-basic}.

\begin{proposition}
\label{prop:twisted-copy}
For every character \(\chi\in\widehat A\), the twisted Weyl-orbit embedding
\(\mathcal T_\chi\) is simultaneously uniformity preserving for every $k$:
\[
\|\mathcal T_\chi(S)\|_{U^k(A)}
=
\|S\|_{Q^k},
\qquad
S\in\mathsf Q,\quad k\geq2.
\]
\end{proposition}

\begin{proof}
Let
\[
M_\chi F(x):=\overline{\chi(x)}F(x).
\]
Then
\[
\mathcal T_\chi=M_\chi\mathcal T,
\qquad
\Delta_h(M_\chi F)
=
\overline{\chi(h)}\,\Delta_hF.
\]
For \(k\geq2\), the scalar factor disappears after one further
discrete derivative, so
\[
\|M_\chi F\|_{U^k(A)}
=
\|F\|_{U^k(A)}.
\]
Therefore, since \(\mathcal T\) is an uniformity preserving,
\[
\|\mathcal T_\chi(S)\|_{U^k(A)}
=
\|\mathcal T(S)\|_{U^k(A)}
=
\|S\|_{Q^k}.
\]
\end{proof}

\begin{theorem}
\label{thm:character-decomposition}
The ambient matrix-valued uniformity space \(\mathsf G\) admits the
decomposition
\[
\mathsf G
=
\bigoplus_{\chi\in\widehat A}
\mathcal T_\chi(\mathsf Q).
\]
Equivalently, every \(F\in\mathsf G\) has the decomposition
\[
F
=
\sum_{\chi\in\widehat A}
\mathcal T_\chi\bigl(\mathcal R_\chi(F)\bigr).
\]
\begin{enumerate}[label=\textup{(\roman*)}]
\item For every $k\geq2$, each summand $\mathcal T_\chi(\mathsf Q)$ is a uniformity-preserving copy of $\mathsf Q$ with respect to the corresponding $k$th uniformity functionals.

\item The
same decomposition is an \(\widehat A\)-grading with respect to the \(C^*\)-algebra structure of \(\mathsf G\)
\end{enumerate}
\end{theorem}

\begin{proof}
For \(\chi,\psi\in\widehat A\), character orthogonality gives
\[
\mathcal R_\chi\mathcal T_\psi
=
\delta_{\chi,\psi}\operatorname{id}_{\mathsf Q},
\]
and hence
\[
\mathcal P_\chi\mathcal P_\psi
=
\delta_{\chi,\psi}\mathcal P_\chi.
\]
Moreover,
\[
\sum_{\chi\in\widehat A}\mathcal P_\chi
=
\mathbb E_{b\in A}
\left(
\sum_{\chi\in\widehat A}\overline{\chi(b)}
\right)\rho_b
=
\rho_0
=
\operatorname{id}_{\mathsf G}.
\]
Therefore
\[
\mathsf G
=
\bigoplus_{\chi\in\widehat A}
\mathcal T_\chi(\mathsf Q).
\]
Part \textup{(i)} now follows from
Proposition~\ref{prop:twisted-copy}.

For \textup{(ii)}, if \(S,T\in\mathsf Q\), then
\[
\mathcal T_\chi(S)\mathcal T_\psi(T)
=
\mathcal T_{\chi\psi}(ST),
\qquad
\mathcal T_\chi(S)^*
=
\mathcal T_{\chi^{-1}}(S^*).
\]
Thus the same decomposition is a \(\widehat A\)-grading of the
\(C^*\)-algebra \(\mathsf G\).
\end{proof}

\begin{remark}
Since the symplectic form on \(A\) is nondegenerate, the map
\[
a\longmapsto \chi_a,
\qquad
\chi_a(x):=\omega^{[a,x]},
\]
identifies \(A\) with its character group \(\widehat A\). Thus the decomposition in
Theorem~\ref{thm:character-decomposition} may equivalently be indexed by
\(a\in A\).
\end{remark}

\end{document}